\documentclass[11pt, reqno]{amsart}

\usepackage{amsmath,amssymb,amsthm, epsfig}
\usepackage{hyperref}
\usepackage{amsmath}
\usepackage{amssymb}
\usepackage{color}
\usepackage{ulem}
\usepackage{epsfig}
\usepackage[mathscr]{eucal}
\usepackage[latin1]{inputenc}

\newtheorem{theorem}{Theorem}
\newtheorem{definition}{Definition}

\newtheorem{lemma}{Lemma}
\newtheorem{proposition}{Proposition}
\newtheorem{corollary}{Corollary}
\newtheorem{remark}{Remark}

\date{}
\numberwithin{equation}{section}
\numberwithin{theorem}{section}
\numberwithin{lemma}{section}
\numberwithin{corollary}{section}
\numberwithin{remark}{section} 
\numberwithin{proposition}{section}
\numberwithin{definition}{section}

\def \R {\mathbb{R}}
\def \supp {\mathrm{supp } }

\def \osc {\mathrm{osc}}

\begin{document}

\title[Fully nonlinear integro-differential equations]{Fully nonlinear integro-differential equations with deforming kernels}

\author[L. Caffarelli]{Luis Caffarelli}
\address{The University of Texas at Austin, Mathematics Department RLM 8.100, 2515 Speedway Stop C1200, Austin, Texas 78712-1202, USA}
\email{caffarel@math.utexas.edu}

\author[R. Teymurazyan]{Rafayel Teymurazyan}
\address{CMUC, Department of Mathematics, University of Coimbra, 3001-501 Coimbra, Portugal}
\email{rafayel@mat.uc.pt}

\author[J.M. Urbano]{Jos\'e Miguel Urbano}
\address{CMUC, Department of Mathematics, University of Coimbra, 3001-501 Coimbra, Portugal}
\email{jmurb@mat.uc.pt}

\thanks{This work was done in the framework of the UT Austin|Portugal CoLab program. LC supported by a DMS-NSF grant. RT and JMU partially supported by FCT -- Funda\c c\~ao para a Ci\^encia e a Tecnologia, I.P., through projects PTDC/MAT-PUR/28686/2017 and UTAP-EXPL/MAT/0017/2017, and grant SFRH/BPD/92717/2013, and by the Centre for Mathematics of the University of Coimbra -- UID/MAT/00324/2013, funded by the Portuguese government through FCT and co-funded by the European Regional Development Fund through Partnership Agreement PT2020.}

\begin{abstract}
We develop a regularity theory for integro-differential equations with kernels deforming in space like sections of a convex solution of a Monge-Amp\`ere equation. We prove an ABP estimate and a Harnack inequality, and derive H\"{o}lder and $C^{1,\alpha}$ regularity results for solutions. 

\bigskip

\noindent \textbf{Keywords:} Integro-differential equations, Monge-Amp\`ere equation,\\ ABP estimate, Harnack inequality, H\"older regularity.

\bigskip

\noindent \textbf{AMS Subject Classifications MSC 2010:} 35R09, 35J60, 35J96, 47G20, 35D40, 35B65.

\end{abstract}

\maketitle

\section{Introduction}\label{s1}

In stochastic control problems (see \cite{OS07}), for example if, in a competitive stochastic game, two players are allowed to choose from different strategies at every step in order to maximize the expected value $u(x)$ at the first exit point of a domain, we encounter the fully nonlinear elliptic integro-differential Isaacs equation 
\begin{equation}\label{1.1}
  Iu(x):=\inf_\alpha\sup_\beta L_{\alpha\beta}u(x)=f(x),
\end{equation}
where
\begin{equation}\label{1.2}
  L_{\alpha\beta}u(x):=\int_{\R^n}(u(x+y)+u(x-y)-2u(x))K_x^{\alpha\beta}(y)\,dy
\end{equation}
are generators of $n$-dimensional pure jump L\'evy processes, those for which diffusion and drift are neglected. The kernels $K_x^{\alpha\beta} (y)$ measure the frequency of jumps in the $y$ direction at the point $x$. For a homogeneous medium, as in \cite{CLU14, CS09} for example, the kernel does not depend on $x$.

We are interested in this paper in the case of a slowly deforming medium for which the level sets of the kernels $K_x^{\alpha\beta}$ are sections of a convex solution $\phi$ of a Monge-Amp\`ere equation. Set
\begin{equation}
v_x(y)=\phi(y)-\phi(x)-\nabla\phi(x)\cdot (y-x),\,\,x,y\in\mathbb{R}^n,
\label{v}
\end{equation}
and note that $v_x\geq0$, since the graph of a convex function stays above supporting hyperplanes.

We will study equation \eqref{1.1} for kernels in \eqref{1.2} that satisfy the bounds
\begin{equation}\label{1.3}
\frac{(2-\sigma)\lambda}{v_x(y)^{\frac{n+\sigma}{2}}}\leq K_x^{\alpha\beta}(y)\leq\frac{(2-\sigma)\Lambda}{v_x(y)^{\frac{n+\sigma}{2}}},
\end{equation}
for constants $\Lambda\geq\lambda>0$ and $\sigma\in(0,2)$. We will assume that kernels are symmetric, i.e., that $K_x^{\alpha\beta}(y)=K_x^{\alpha\beta}(-y)$, but this is a purely technical assumption, rendering the proofs simpler. The kernels in \eqref{1.3} may be very degenerate, since the sections of $\phi$ are comparable to ellipsoids, which may have very degenerate eccentricity (see \cite{G16}, for example). The right hand side of \eqref{1.1} is assumed to be a bounded and continuous function in $\R^n$. The case of $\phi(y)=|y|^2$ resembles the kernel of the fractional Laplacian and was studied in \cite{CS09}, extending the notion of ellipticity by means of the relations
$$
M^-w(x)\leq I(u+w)(x)-Iu(x)\leq M^+w(x),
$$
where
$$
M^-u(x):=\inf_{L\in\mathcal{L}}Lu(x)
$$
and
$$
M^+u(x):=\sup_{L\in\mathcal{L}}Lu(x)
$$
are analogs of the extremal Pucci operators. Here $\mathcal{L}$ is the class of operators of \eqref{1.2} type whose kernels satisfy \eqref{1.3}.  

The approach used in \cite{CS09} is a non-local version of the strategy used in \cite{CC95}. In our case, the strong degeneracy of the kernels precludes the use of standard covering arguments, a difficulty that can be overcome considering the deformation of the kernels is driven by the Monge-Amp\`ere geometry. This is due to the fact that sections of a convex solution of a Monge-Amp\`ere equation enjoy an engulfing property: if two sections overlap, then by lifting one by a universal constant, we engulf the other. After renormalization, sections become comparable to balls (see \cite{F17, G16}) and this geometry allows for a refinement of the known techniques in order to develop a regularity theory. 

The regularity theory in the classical non-variational approach (see \cite{CC95}) heavily depends on the Aleksandrov-Bakelman-Pucci (ABP) estimate
$$
\sup_{B_1}u\leq C(n)\left(\int_{\{\Gamma=u\}\cap B_1}(f^+)^n\right)^{1/n},
$$
where $u$ is a viscosity subsolution of the maximal Pucci equation with $(-f)$ as the right-hand side, which is non-positive outside the unit ball $B_1$, and where $\Gamma$ is the concave envelope of $u$ in $B_3$. The ABP estimate bridges the gap between a pointwise estimate and an estimate in measure. For $u(0)\geq1$, it provides the bounds
$$
1\leq C\|f\|_\infty|\{\Gamma=u\}\cap B_1|^{1/n}\leq C\|f\|_\infty|\{u\geq0\}\cap B_1|^{1/n},
$$
where $|E|$ stands for the $n$-dimensional Lebesgue measure of the set $E$.

We need a nonlocal version of the ABP estimate like the one in \cite{CLU14,CS09} but we are  dealing with kernels which deform in space and we must have some control over the deformation. It turns out that if the deformation is driven by the Monge-Amp\`ere geometry then the engulfing property of the Monge-Amp\`ere sections provides the needed environment to use a covering lemma from \cite{CG96} and obtain a variant of the ABP estimate. Once this is achieved, we can use a variant of Calder\'on-Zygmund decomposition from \cite{CG97} to obtain the Harnack inequality and further regularity results. The heart of the proof is to find the suitable geometry of the neighbourhoods of the contact points within which there is a portion where a sub-solution $u$ stays quadratically close to the tangent plane of its concave envelope $\Gamma$ and such that, in smaller neighbourhoods, $\Gamma$ has quadratic growth. This task, in turn, requires a certain control over the deformation of sections, that allows one to properly define a suitable concave envelope. We then conclude that if a concave function stays below its tangent plane translated by $-r$ (for a given number $r>0$) in a portion of an annulus of the unit section, then $\Gamma+r$ stays above its tangent plane in the interior section of the annulus. Through the normalisation map, we ultimately extend the regularity theory from \cite{CS09} into the framework of slowly deforming kernels.

The paper is organized as follows: in Section \ref{s2}, we list several known facts about the behaviour of a convex solution of the Monge-Amp\`ere equation. It is also here that we define the analogs of the Pucci extremal operators in our framework and state some preliminary results. Section \ref{s3} is devoted to the ABP estimate. Using properties of the level sets of the kernels (sections of a convex solution of the Monge-Amp\`ere equation) and a covering argument from \cite{CG96}, we get a version of the ABP estimate in a non-local setting with slowly deforming kernels (Corollary \ref{c3.2} and Theorem \ref{l3.1}). In Section \ref{s4}, we construct an auxiliary function which is a subsolution of the minimal equation outside of a small section and is strictly positive in a larger section (Lemma \ref{l4.2}). This function is added to $u$ in Section \ref{s5} to force the contact set with $\Gamma$ to stay inside intermediate normalized sections. In this way, using a variant of Calder\'on-Zygmund decomposition from \cite{CG97}, we are able to prove a variant of the so-called $L^\varepsilon$ estimate, which bridges the gap between a pointwise estimate and an estimate in measure (Theorem \ref{t5.1}), the main step towards the Harnack inequality, which we prove in Section \ref{s6} (Theorem \ref{t6.1}). Finally, as a consequence of the Harnack inequality, we derive H\"{o}lder (Theorem \ref{t6.2}) and $C^{1,\alpha}$ (Theorem \ref{t6.3}) regularity results for solutions.

\section{Preparatory material}\label{s2}
In this section we list several known facts about the sections of a convex solution of a Monge-Amp\`ere equation (which can be found in \cite{AFT98,C901,C902,C91,CG97,F17,G16,GH00}; see also \cite{MS17}, where a non-local linearized Monge-Amp\`ere equation is treated, as well as \cite{BL02,CL12,SS16} for a more comprehensive approach to non-local equations) and also state some preliminary results.
 
\subsection{Sections of a Monge-Amp\`ere convex solution}

To understand the deformation of kernels, we need to look at the sections of a $C^{2}$ convex solution of the Monge-Amp\`ere equation
$$
\det D^2\phi =g,
$$
where $0<g_-\le g(x)\le g^+<\infty$, $x\in\R^n$, for constants $g_-$ and $g^+$. A section $S_r(x)$ of $\phi$ is defined as
$$
S_r(x):=\{y\, : \, \phi(y)<\phi(x)+\nabla\phi(x)\cdot(y-x)+r^2\},
$$
or, recalling \eqref{v},
\begin{equation}\label{2.1}
S_r(x)=\{y \, : \,v_x(y)<r^2\}.
\end{equation}
Geometrically, it amounts to taking a supporting hyperplane at $x$ and lifting it by $r^2$ to carve out the convex set $S_r$. These are the ``balls of radius $r$'' in the Monge-Amp\`ere geometry. The sections are ``balanced'' around the point from which we lift, and we know precisely how the volume grows. Before proceeding, we point out that our results depend upon the geometry of Monge-Amp\`ere sections and are valid for generic $x$ and $r$ rather than a specific section. For simplicity, we will often omit the center of the section, when it is not playing an essential role in the proofs.

The proof of following theorem can be found in \cite{F17,G16}.
\begin{theorem}\label{t2.1}
There is an ellipsoid $E$ of volume $r^n$ such that
$$
cE\subset S_r(x)\subset CE,
$$
where $c$ and $C$ are universal positive constants depending only on $n$.
\end{theorem}
 
Since $E$ is an ellipsoid, there is an affine transformation $T$ such that $T(E)=B_1$, and therefore
$$
B_{\alpha_n}\subset T(S_r)\subset B_1,
$$
with $B_r$ being the ball of radius $r$ centered at $0$, and where $\alpha_n$ is a positive constant depending only on $n$. We will refer to $T$ as a normalization map that normalizes the section $S_r$, and to $T(S_r)$ as a normalized section. 

We list several properties of the sections for future reference. The first fact is that sections of $\phi$ satisfy the engulfing property. More precisely, if two sections of similar size overlap, a universal multiple of one engulfs the other, indicating that they must have roughly the same shape. The proof can be found in \cite{F17,G16}.
\begin{theorem}\label{t2.2}
There is a universal constant $\gamma>1$ such that if $y\in S_r(x)$, then
$$
S_r(x)\subset S_{\gamma r}(y).
$$
\end{theorem}

The next theorem provides a quantitative estimate on the size of normalized sections. It states that if an affine map normalizes a section, then all other sections with comparable height are still comparable to balls (see \cite{AFT98,F17,G16} for the proof).
\begin{theorem}\label{t2.3} Let $T$ be an affine transformation that normalizes $S_R(x)$ and $r\leq R$. If
$$
S_R(x)\cap S_r(y)\neq\emptyset,
$$
then there exist positive constants $K_1$, $K_2$, $K_3$, and $\varepsilon$, such that 
$$
B_{K_2\frac{r}{R}}(Ty)\subset T(S_r(y))\subset B_{K_1(\frac{r}{R})^\varepsilon}(Ty),
$$
and $Ty\in B_{K_3}(0)$.
\end{theorem}
As a consequence of the previous theorem, we have a result on the deformation of sections, the proof of which can be found in \cite{CG96,CG97,F17,G16}.
\begin{theorem}\label{t2.4}
The following assertions hold:
\begin{itemize}
\item[(i)] there exist $C_0>0$ and $p_0\geq1$ such that whenever $0<r<s\leq1$, $t>0$ and $x\in S_{rt}(y)$, then
$$
S_{C_0t(s-r)^{p_0}}(x)\subset S_{st}(y);
$$
\item[(ii)] there exist $C_1>0$ and $p_1\geq1$ such that whenever $0<r<s<1$, $t>0$ and $x\in S_t(y)\setminus S_{st}(y)$, then
$$
S_{C_1t(s-r)^{p_1}}(x)\cap S_{rt}(y)=\emptyset;
$$
\end{itemize}
and as a consequence
\begin{itemize}
\item[(iii)] there exists $\delta>0$ such that whenever $x\in S_{3t/4}(y)\setminus S_{t/2}(y)$, then
$$
S_{\delta t}(x)\subset S_t(y)\setminus S_{t/4}(y).
$$
\end{itemize}
Also, for $r>0$,
\begin{itemize}
\item[(iv)] $|S_r(x)|\leq 2^n|S_{r/2}(x)|$;
\item[(v)] $|S_r(x)|-|S_{\epsilon r}(x)|\leq n(1-\epsilon)|S_r(x)|$, for all $\epsilon\in(0,1)$.
\end{itemize}
\end{theorem}
The following Besicovitch type covering lemma is from \cite[Lemma 1]{CG96}. It plays an essential role in our analysis.
\begin{lemma}\label{l2.0}
If $A\subset\R^n$ is a bounded set and $S_r(x)$, $x\in A$, $r\leq C$ for a fixed constant $C>0$, is a covering of $A$, then there is a countable subcovering such that
\begin{itemize}
\item[(i)] $A\subset\bigcup_{k=1}^\infty S_{r_k}(x_k)$;
\item[(ii)] $x_k\notin\bigcup_{j<k}S_{r_j}(x_j)$, $\forall k\geq2$;
\item[(iii)] for small $\epsilon>0$, the family $\{S_{(1-\epsilon)r_k}(x_k)\}_{k=1}^\infty$ has bounded overlaps, i.e., there exists a universal constant $M>0$ such that
$$
\sum_{k=1}^\infty\chi_{S_{r_k(1-\epsilon)}(x_k)}(x)\leq M\log\frac{1}{\epsilon},
$$
\end{itemize}
where $\chi_E$ is the characteristic function of the set $E$.
\end{lemma}

The next covering lemma is from \cite[Theorem 3]{CG97}. It is a variant of the Calder\'on-Zygmund decomposition and is used to derive the so-called $L^\varepsilon$ estimate, giving access to the Harnack inequality.
\begin{lemma}\label{l2.1}
If $A$ is an open, bounded set and $\theta\in(0,1)$, then there exists a family of sections $\left\{ S_{r_k}(x_k) \right\}_{k=1}^\infty$ such that
\medskip
\begin{itemize}
\item[(1)] $\{x_k\}_{k=1}^\infty\subset A$;
\medskip
\item[(2)] $A\subset\displaystyle\bigcup_{k=1}^\infty S_{r_k}(x_k)$;
\medskip
\item[(3)] $\displaystyle \frac{\left| A\cap S_{r_k}(x_k)\right|}{\left| S_{r_k}(x_k) \right|}=\theta$;
\medskip
\item[(4)] $|A| <c(\theta) \left| \displaystyle\bigcup_{k=1}^\infty S_{r_k}(x_k)\right|$, 

\end{itemize}
where $c(\theta)\in(0,1)$ is a constant depending only on $\theta$, but not on $A$ nor the family of sections.
\end{lemma}
We will make use of the normalization map to normalize sections when needed. As the normalization of a section implies that all the other sections are comparable to Euclidian balls, our results are valid for generic $r$, which comes at the price of constants, depending upon the normalization map, appearing in the estimates (see Theorem \ref{t6.2}, for example).

\subsection{Extremal operators}
Note that when $g\equiv1$, $\phi$ is a quadratic polynomial, and we are back to the case of the fractional Laplacian, studied in \cite{CS09}. We also point out that when $\lambda=\Lambda=1$ in \eqref{1.3}, we are in the framework of the fractional non-local linearized Monge-Amp\`ere equation, studied in \cite{MS17}.

Setting
$$
\delta(u,x,y):=u(x+y)+u(x-y)-2u(x),
$$
$L_{\alpha\beta}$ can be rewritten as
$$
L_{\alpha\beta}u(x)=\int_{\R^n}\delta(u,x,y)K_x^{\alpha\beta}(y)\,dy.
$$
We now define the adequate class of test functions for our purposes.
\begin{definition}
A function $\varphi$ is said to be $C^{1,1}$ at a point $x$, and we write $\varphi\in C^{1,1}(x)$, if there exist $v\in\R^n$ and $M, \eta_0>0$ such that
$$
\left| \varphi(x+y)-\varphi(x)-v\cdot y \right| \leq M |y|^2,
$$
for every $|x|<\eta_0$. A function $\varphi$ is said to be $C^{1,1}$ in a set $\Omega$, and we write $\varphi\in C^{1,1}(\Omega)$, if it is $C^{1,1}$ at every point in $\Omega$, for a uniform constant $M$.
\end{definition}

Since solutions of the Monge-Amp\`ere equation with a bounded right-hand side have quadratic growth when a section is normalized, the kernels in our framework are a deformation of a kernel comparable to that of the fractional Laplacian. Hence, throughout the paper, we will use the normalization map to make sections comparable to Euclidean balls, and then change the variables back. In this way, we reproduce several properties in our framework. However, this approach may result in the dependence of the constants appearing in the estimates on the normalization map (see the proof of Theorem \ref{t6.2} for details).

\begin{remark}\label{r2.2}
Let $u\in C^{1,1}(x)\cap L^\infty(\R^n)$, then $Iu(x)\in\R$ (see Remark 2.2 of \cite{CLU14}).
\end{remark}
\begin{definition}\label{d2.2}
Let $f$ be a bounded and continuous function in $\R^n$. A function $u:\R^n\rightarrow\R$, upper continuous in $\overline{\Omega}$, is a viscosity subsolution of the equation $Iu=f$, and we write $Iu\geq f$, if whenever $x_0\in\Omega$, $B_r(x_0)\subset\Omega$, for some $r$, and $\varphi \in C^2(\overline{B_r(x_0)})$ satisfies
$$\varphi(x_0)=u(x_0) \quad {\rm and} \quad \varphi(y)>u(y), \ \forall y\in B_r(x_0) \setminus \{x_0\},$$
then, if we let
$$
v:=
\left\{
\begin{array}{ll}
\varphi & \hbox{in } \ B_r(x_0)\\
\\
u & \hbox{in } \ \R^n\setminus B_r(x_0),
\end{array}
\right.
$$
we have $Iv(x_0)\geq f(x_0)$. 

A viscosity supersolution is defined analogously and a function is called a viscosity solution if it is both a viscosity subsolution and a viscosity supersolution.
\end{definition}

\begin{remark}\label{r2.3}
Functions which are merely $C^{1,1}$ at a contact point $x$ can be used as test functions in the definition of viscosity solution (see Lemma 4.3 in \cite{CS09}).
\end{remark}

Let $\mathcal{L}$ be the collection of linear operators $L_{\alpha\beta}$ satisfying \eqref{1.3}. We define the maximal and minimal operators (the Pucci analogs) with respect to the class $\mathcal{L}$ as
$$
M^+u(x):=\sup_{L\in\mathcal{L}}Lu(x)
$$
and
$$
M^-u(x):=\inf_{L\in\mathcal{L}}Lu(x).
$$
By definition, if $M^+u(x)<\infty$ and $M^-u(x)<\infty$, we have the simple forms
$$
M^+u(x)=(2-\sigma)\int_{\R^n}
\frac{\Lambda\delta^+-\lambda\delta^-}{v_x(y)^{\frac{n+\sigma}{2}}}\,dy
$$
and
$$
M^-u(x)=(2-\sigma)\int_{\R^n}
\frac{\lambda\delta^+-\Lambda\delta^-}{v_x(y)^{\frac{n+\sigma}{2}}}\,dy,
$$
where $\delta^+$ and $\delta^-$ are, respectively, the positive and negative parts of $\delta$. 

\begin{definition}\label{d2.3}
	The operator $I$ is called elliptic with respect to the class $\mathcal{L}$ of integro-differential operators, if
	\begin{itemize}
		\item $Iu(x)$ is well defined for all $u\in C^{1,1}(x)$, $u$ bounded;
		\item $Iu\in C(\Omega)$ once $u\in C^2(\Omega)$;
		\item  $M^-(u-v)(x)\le Iu(x)-Iv(x)\le M^+(u-v)(x)$, for any bounded functions $u$ and $v$ which are $C^{1,1}$ at $x$.
	\end{itemize}
\end{definition}

We close this section by recalling several results, the proofs of which can be derived as in \cite{CS09}. The first result says that if $u$ can be touched from above, at a point $x$, with a paraboloid, then $Iu(x)$ can be evaluated classically.
\begin{lemma}\label{l2.2}
  If $Iu\geq f$ in $\Omega$, and $\varphi\in C^2$ touches $u$ from above at a point $x\in\Omega$, then $Iu(x)$ is defined in the classical sense, and $Iu(x)\geq f(x)$.
\end{lemma}
Another important result is the continuity of $Iv$ in $\Omega$, if $v\in C^{1,1}(\Omega)$.
\begin{lemma}\label{l2.3}
  If $v$ is a bounded function in $\R^n$, which is $C^{1,1}$ in some open set $\Omega$, then $Iv\in C(\Omega)$.
\end{lemma}

Although in general one can not compare two solutions at a given point, since they may not have the required behaviour simultaneously, it is possible to show (see \cite[Section 5]{CS09}) that the difference of a subsolution of the maximal operator and a supersolution of the minimal operator is a subsolution of the maximal operator.  
\begin{lemma}\label{l2.4}
  If $\Omega$ is an open, bounded set, and $u$ and $v$ are bounded functions in $\R^n$ such that
  \begin{enumerate}
    \item $u$ is upper-semicontinuous, $v$ is lower-semicontinuous in $\overline{\Omega}$;
    \item $Iu\geq f$, $Iv\leq g$ in the viscosity sense in $\Omega$ with $f$, $g$ continuous,
  \end{enumerate}
  then
  $$
  M^+(u-v)\geq f-g\,\hbox{ in }\,\Omega
  $$
  in the viscosity sense.
\end{lemma}

As in \cite{CS09}, Lemma \ref{l2.4} leads to the following comparison principle. 
\begin{theorem}\label{t2.5}
	Let $I$ be elliptic with respect to a class $\mathcal{L}$, $\Omega\subset\R^n$ be a bounded open set, $u$, $v$ be bounded functions in $\R^n$ such that $u$ is upper semi-continuous in $\overline{\Omega}$ and $v$ is lower semi-continuous in $\overline{\Omega}$. If $Iu\ge f$ and $Iv\le f$ in $\Omega$, where $f$ is continuous, and $u\le v$ in $\R^n\setminus\Omega$, then $u\le v$ in $\Omega$.
\end{theorem}	
	
The existence of a solution for the Dirichlet problem then follows from the comparison principle, by constructing suitable barriers and using Perron's method (see \cite{I89}).

\section{The ABP estimate}\label{s3}
In this section we prove a version of the ABP estimate, which will give access to the regularity theory. We start with the following proposition, which then allows one to properly define a suitable concave envelope for functions.
\begin{proposition}\label{p3.1}
Let $\gamma>1$ be the engulfing constant from Theorem \ref{t2.2}. If $x\in S_1(0)$, then there exists a constant $\tau>\gamma$ such that whenever either $x+y$ or $x-y$ is not in $S_\tau(0)$, then both of them are not in $S_1(0)$.
\end{proposition}
\begin{proof}
Let $x+y\notin S_\tau(0)$, for some $\tau>\gamma$ to be chosen later. We want to show that $x-y\notin S_1(0)$. We argue by contradiction and assume that $x-y\in S_1(0)$. By the engulfing property, Theorem \ref{t2.2}, this implies that $S_1(0)\subset S_\gamma(x-y)$, and therefore
$$
S_\gamma(x-y)\cap S_\tau(0)\neq\emptyset,
$$
since they both contain $S_1(0)$. If $T$ is an affine transformation that normalizes the section $S_\tau(0)$, i.e.,
$$
B_{\alpha_n}\subset T(S_\tau(0))\subset B_1,
$$
then from Theorem \ref{t2.3} we obtain that
$$
T(S_\gamma(x-y))\subset B_{K_1(\frac{\gamma}{\tau})^\varepsilon}(Tx-Ty),
$$
for some positive constants $K_1$ and $\varepsilon$. Since $0\in S_\gamma(x-y)$, the above inclusion then gives
$$
|Tx-Ty|<K_1\left(\frac{\gamma}{\tau}\right)^\varepsilon.
$$
Similarly, since also $x\in S_\gamma(x-y)$, then the above inclusion provides
$$
|Ty|<K_1\left(\frac{\gamma}{\tau}\right)^\varepsilon,
$$
hence
$$
|Tx|<2K_1\left(\frac{\gamma}{\tau}\right)^\varepsilon.
$$
Combining the last two inequalities, we obtain
$$
|Tx+Ty|<3K_1\left(\frac{\gamma}{\tau}\right)^\varepsilon.
$$
On the other hand, since $x+y\notin S_\tau(0)$, then $Tx+Ty\notin B_{\alpha_n}$, i.e., 
$$|Tx+Ty|\geq\alpha_n.$$
By choosing $\tau>\gamma\left(\frac{3K_1}{\alpha_n}\right)^{1/\varepsilon}$, we get a contradiction. The other case is proved analogously.
\end{proof}
Hereafter, we will assume that $\tau>\gamma$ is as in Proposition \ref{p3.1}. Whenever the center of a section is the origin, we will omit it, i.e., we will write $S_r$ instead of $S_r(0)$.

Let $u$ be a non-positive function outside the section $S_1$. The concave envelope of $u$ is defined by
$$\Gamma(x):=
\left\{
\begin{array}{ll}
\min \left\{ p(x) \, : \, p\hbox{ is a plane and }p\geq u^+\hbox{ in }S_\tau \right\}  & \hbox{ in }S_\tau \\
0 & \hbox{ in }\R^n\setminus S_\tau.
\end{array}
\right.
$$
\begin{lemma}\label{l3.1}
Let $u \leq 0$ in $\R^n\setminus S_1$ and $\Gamma$ be its concave envelope. If $M^+u(x)\geq-f(x)$ in $S_1$, then there is a constant $C_0>0$, depending only on $\lambda$ and $n$ (but not on $\sigma$), such that, for any $x\in\{u=\Gamma\}\cap S_1$ and any $M>0$, there exists $k$ such that
$$
|W_k(x)|\leq C_0\frac{f(x)}{M}|R_k(x)|,
$$  
where $R_k(x)=S_{r_k}(x)\setminus S_{r_{k+1}}(x)$, $r_k=2^{-1/(2-\sigma)-k}$ and
$$
W_k(x):=R_k(x)\cap\left\{y\,:\, u(y)<u(x)+\nabla\Gamma(x)\cdot (y-x)-Mr_k^2\right\}.
$$
Here $\nabla\Gamma$ stands for any element of the superdifferential of $\Gamma$ at $x$, which will coincide with its gradient when $\Gamma$ is differentiable. 
\end{lemma}

\begin{proof} Since $u$ can be touched by the plane
$$
\Gamma(x)+\nabla\Gamma(x)\cdot(y-x)
$$
from above at $x$, then from Lemma \ref{l2.2}, $M^+u(x)$ is defined classically, and we have
$$
M^+u(x)=(2-\sigma)\int_{\R^n}
\frac{\Lambda\delta^+-\lambda\delta^-}{v_x^{\frac{n+\sigma}{2}}(y)}\,dy.
$$
Note that $\delta(u,x,y)=u(x+y)+u(x-y)-2u(x)\leq0$, whenever $x\in\{u=\Gamma\}$. In fact, if both $x-y$ and $x+y$ are in $S_\tau$, then $\delta\leq0$, since $u(x)=\Gamma(x)=p(x)$ for some plane $p$ that remains above $u$ in the whole section $S_\tau$. On the other hand, if either $x-y$ or $x+y$ is not in $S_\tau$, then by Proposition \ref{p3.1} both are not in $S_1$, and thus $u$ is non-positive at those points. Hence, in any case, $\delta\leq0$, and therefore
\begin{eqnarray*}
-f(x)\leq M^+u(x)&=&(2-\sigma)\int_{\R^n}
\frac{-\lambda\delta^-}{v_x(y)^{\frac{n+\sigma}{2}}}\,dy\nonumber\\
&\leq&(2-\sigma)\int_{S_{r_0}(x)}\frac{-\lambda\delta^-}{v_x(y)^{\frac{n+\sigma}{2}}}\,dy,
\end{eqnarray*}
where $r_0=2^{-1/(2-\sigma)}$. Now, splitting the integral in the sections and reorganizing terms, we obtain
$$
f(x)\geq(2-\sigma)\lambda\sum_{k=0}^\infty\int_{S_{r_k}(x)\setminus S_{r_{k+1}}(x)}\frac{\delta^-}{v_x(y)^{\frac{n+\sigma}{2}}}\,dy,
$$
which, together with \eqref{2.1}, provides
\begin{equation}\label{3.1}
f(x)\geq(2-\sigma)\lambda\sum_{k=0}^\infty\int_{R_{k}(x)}\frac{\delta^-}{{r_k^{n+\sigma}}}\,dy.
\end{equation}
Note that since $x\in\{u=\Gamma\}$, then
\begin{equation}\label{3.2}
W_k(x)\subset R_k(x)\cap\{z \, : \, -\delta>2Mr_k^2\}.
\end{equation}
But $\delta\leq0$ and so $-\delta=\delta^-$. From \eqref{3.1}-\eqref{3.2}, we then have
\begin{equation}\label{3.3}
f(x)\geq2M\lambda(2-\sigma)\sum_{k=0}^\infty r_k^{2-n-\sigma}|W_k(x)|.
\end{equation}
Suppose now the conclusion of the lemma is false. Then \eqref{3.3} implies
\begin{equation}\label{3.4}
f(x)\geq2\lambda C_0(2-\sigma)f(x)\sum_{k=0}^\infty r_k^{2-n-\sigma}|R_k(x)|.
\end{equation}
Using Theorem \ref{t2.1}, we estimate
$$
|R_k(x)|\geq cr_k^n,
$$
where $c>0$ is a universal constant. Combining the latter with \eqref{3.4}, we deduce
\begin{eqnarray*}
f(x)&\geq& 2\lambda(2-\sigma)C_0cf(x)\sum_{k=0}^\infty r_k^{2-\sigma}\nonumber\\
&\geq&C(2-\sigma)\frac{1}{1-2^{-(2-\sigma)}}C_0f(x)\nonumber\\
&\geq&CC_0f(x),
\end{eqnarray*}
where the last inequality holds because $(2-\sigma)/(1-2^{-(2-\sigma)})$ remains bounded below for $\sigma\in(0,2)$. The constant $C>0$ depends only on $\lambda$, $n$ but not on $\sigma$. By choosing $C_0$ large enough, we obtain a contradiction.
\end{proof}

\begin{remark}\label{r3.1}
Note that if $M^+u(x)\geq g(x)$, then $u\neq\Gamma$ in $\{g>0\}$.
\end{remark}

The next lemma reveals that Lemma \ref{l3.1} implies a uniform quadratic detachment of $\Gamma$ from its tangent plane in a smaller section.

\begin{lemma}\label{l3.2}
Let $r\in(0,1)$, $\Gamma$ be a concave function in $S_r(x)$ and $h>0$. There exists $\varepsilon_0>0$ such that, if 
\begin{align*}
& \left| \left( S_r \setminus S_{r/2} \right) (x)\cap\{y\, : \, \Gamma (y)<\Gamma(x)+\nabla\Gamma(x)\cdot (y-x)-h\} \right| \\
& \leq \varepsilon \left| S_r\setminus S_{r/2} \right|,
\end{align*}
for $0<\varepsilon\leq\varepsilon_0$, then
$$
\Gamma(y)\geq\Gamma(x)+\nabla\Gamma(x)\cdot (y-x)-h,
$$
in the whole section $S_{r/2}(x)$.
\end{lemma}

\begin{proof}
Let $y\in S_{r/2}(x)$. Using Theorem \ref{t2.4}, we conclude that there exist two points $z_1$ and $z_2$ in $S_r(x)\setminus S_{r/2}(x)$ such that the sections $S_{cr}(z_1)$ and $S_{cr}(z_2)$ are contained in the ring $S_r(x)\setminus S_{r/2}(x)$, for some constant $c>0$. Moreover, we can choose these points such that $y=\alpha z_1+(1-\alpha)z_2$ for some $\alpha\in(0,1)$. If $\varepsilon_0$ is small enough, then at those points one has 
$$
\Gamma(z_1)\geq\Gamma(x)+\nabla\Gamma(x)\cdot (z_1-x)-h
$$
and
$$
\Gamma(z_2)\geq\Gamma(0)+\nabla\Gamma(x)\cdot (z_2-x)-h.
$$
The concavity of $\Gamma$ then gives 
$$
\Gamma(y)\geq\alpha\Gamma(z_1)+(1-\alpha)\Gamma(z_2)=\Gamma(x)+\nabla\Gamma(x)\cdot (y-x)-h.
$$
\end{proof}

\begin{corollary}\label{c3.1}
Let $u$ be as in Lemma \ref{l3.1} and $r_0=2^{-1/(2-\sigma)}$. Under the hypothesis of Lemma \ref{l3.1}, for every $\varepsilon >0$, there exist $C=C(n,\varepsilon)>0$ and $r\in(0,r_0)$ such that
\begin{align*}
& \left| (S_r\setminus S_{r/2})(x)\cap\{ y \, :\, u(y)<u(x)+\nabla\Gamma(x)\cdot (y-x)-Cf(x)r^2\} \right| \\
& \leq  \varepsilon \left| S_r\setminus S_{r/2} \right|
\end{align*}
and
$$
|\nabla\Gamma(S_{r/4}(x))|\leq Cf(x)^n|S_{r/4}(x)|.
$$
\end{corollary}
\begin{proof}
By taking $M=C_0f(x)/\varepsilon$ in Lemma \ref{l3.1}, we obtain the first estimate with $C=C_1:=C_0/\varepsilon$. Moreover, since $u(x)=\Gamma(x)$ and $u(y)\leq\Gamma(y)$, for $y\in S_r(x)$, one has
$$
\left( S_r\setminus S_{r/2} \right) (x) \cap\left\{ y\,:\,\Gamma(y)<\Gamma(x)+\nabla\Gamma(x)\cdot (y-x)-C_1f(x)r^2 \right\} \subset W_r(x),
$$
where
$$
W_r(x):= \left( S_r\setminus S_{r/2} \right)(x)\cap\left\{y\,:\,u(y)<u(x)+\nabla\Gamma(x)\cdot (y-x)-C_1f(x)r^2\right\}.
$$
Set
$$
F(y):=\Gamma(y)-\Gamma(x)-\nabla\Gamma(x)\cdot (y-x)+C_1f(x)r^2.
$$
From Lemma \ref{l3.2} and the concavity of $\Gamma$, we have
\begin{equation}\label{3.6}
0\leq F(y)\leq C_1f(x)r^2, \quad \forall y \in S_{r/2}(x).
\end{equation}
Since $F$ is concave and
$$
\nabla F(y)=\nabla\Gamma(y)-\nabla\Gamma(x),
$$
using \eqref{3.6}, we obtain the bound
$$
|\nabla\Gamma(y)-\nabla\Gamma(x)|\leq C_2f(x)r^2,\quad \forall y \in S_{r/4}(x),
$$
for a constant $C_2>0$. 
Therefore,
$$
\nabla\Gamma(S_{r/4}(x))\subset B_{C_2f(x)r^2}(\nabla\Gamma(x))
$$
and, estimating the measures of these sets and using Theorem \ref{t2.1}, we obtain, with $\alpha (n)$ denoting the volume of the unit ball and observing that $r^2 <r$,
$$
|\nabla\Gamma(S_{r/4}(x))|\leq \alpha(n) C_2^nf(x)^nr^{2n}\leq C_3f(x)^n \left |S_{r/4}(x) \right|,
$$
for a constant $C_3>0$. Taking $C=\max\{C_1,C_3\}$, we conclude the proof.
\end{proof}
We then derive a lower bound on the volume of the union of the sections $S_r$, where $\Gamma$ (and $u$) detaches quadratically from its tangent plane. 

\begin{corollary}\label{c3.2}
For each $x\in\Sigma:=\{u=\Gamma\}\cap S_1$, let $S_r(x)$ be the section obtained in Corollary \ref{c3.1}. Then
$$
C(\sup u)^n\leq\left|\displaystyle\bigcup_{x\in\Sigma}S_r(x)\right|.
$$
\end{corollary}
\begin{proof}
Using Lemma \ref{l2.0}, we cover $\Sigma$ by sections $S_r$ with bounded overlaps. Since $\Gamma$ has quadratic growth in each section $S_k$ of the covering, then from Corollary \ref{c3.1} we have
$$
\left| \nabla\Gamma(S_k) \right| \leq C \left| S_k \right|,
$$
where $C>0$ is a universal constant. Strictly speaking, the estimate is only valid on the set $S_{r/4}$. A rigorous justification follows in the same manner as in \cite{SS16}, as Lemma 4.5 is used to prove Lemma 4.1. Thus,
\begin{eqnarray*}
(\sup u)^n&=& \left( \sup\Gamma \right)^n  \leq  C \left| \nabla\Gamma \left( S_\tau \right) \right|  =  C \left| \nabla\Gamma \left( \Sigma \right) \right| \\
& \leq & C \sum_k \left| \nabla\Gamma \left( S_k \right) \right| \\
& \leq & C \sum_k \left| S_k \right|.
\end{eqnarray*}
\end{proof}

The next result is a consequence of Corollary \ref{c3.1}, and provides the first step towards the so-called $L^\varepsilon$ estimate.

\begin{theorem}\label{t3.1}
There exists  a constant $\kappa>0$ and a countable family of sections $\left\{S_i\right\}_{i=1}^\infty$, with center $x_i\in\Sigma$ and height $\frac{r}{4}\leq r_i<\frac{r}{2}$, where $r\in(0,r_0)$ is as in Corollary \ref{c3.1}, covering $\Sigma$ and with bounded overlaps, such that 
$$
|\nabla\Gamma(\overline{S}_i)|\leq C\left( \displaystyle\max_{\overline{S}_i}f \right)^n \left| S_i \right|
$$
and
$$
\left| \left\{ y\in\kappa S_{i} \, : \, u(y)\geq\Gamma(y)-C \left( \displaystyle\max_{\overline{S}_i}f \right)r_i^2 \right\} \right| \geq \mu \left| S_i \right|,
$$
where the constants $C>0$ and $\mu>0$ depend only on $n$, $\lambda$, $\Lambda$, but not on $\sigma$.
\end{theorem}

\begin{proof}
Using Lemma \ref{l2.0}, we find the covering $\left\{S_i(x_i)\right\}_{i=1}^\infty$ satisfying the desired properties.   We have $\overline{S}_i\subset S_{r/2}(x_i)$ and, by Theorem \ref{t2.4}, there is a constant $\kappa>0$ such that 
$$S_r(x_i)\subset\kappa S_i.$$
Moreover, since $\Gamma$ is concave, we also have
$$\Gamma(y)\leq u(x_i)+\nabla\Gamma(x_i)\cdot(y-x_i).$$

From Corollary \ref{c3.1} and the fact that $r_i$ and $r$ are comparable (recall also that the volume of $S_r$ is comparable to $r^n$), we obtain
\begin{align*}
&\left|\left\{ y\in\kappa S_i\,:\,u(y)\geq\Gamma(y)-C \left( \displaystyle\max_{\overline{S}_i}f \right)r_i^2 \right\} \right| \\
&\geq \left| \left\{ y\in\kappa S_i\,:\,u(y)\geq u(x_i)+\nabla\Gamma(x_i)\cdot(y-x_i)-Cf(x_i)r^2 \right\} \right|\\
&\geq \left(1-\varepsilon \right) \left| S_r\setminus S_{r/2} \right|\\
& \geq\mu \left| S_i \right|.
\end{align*}
\end{proof}

\section{An auxiliary function}\label{s4}
In order to prove the Harnack inequality, one needs to show that under the hypothesis of Lemma \ref{l3.1}, $u$ is non-negative, not just in a positive portion of section $S_1$, but in a positive portion of any middle-sized section centered in a smaller section $S_{r}\subset S_1$. Having in mind the localization of the contact set, we construct a function which is a subsolution of the minimal equation outside of a small section and is strictly positive in a larger section. This function will later be added to $u$ to force the contact set with $\Gamma$ to stay inside of the intermediate sections. 
\begin{lemma}\label{l4.1}
For a given $R>1$, there exist $m>0$ and $\sigma_0\in(0,2)$ such that the function
$$
F(x):=\min\left(2^m,|x|^{-m}\right)
$$
satisfies
$$
M^-F(x)\geq0,
$$
for every $\sigma\in(\sigma_0,2)$ and $1\leq|x|\leq R$. The constants $m$ and $\sigma_0$ depend only on $\lambda$, $\Lambda$, $R$ and dimension.
\end{lemma}
\begin{proof}
Without loss of generality, it is enough to prove the lemma for the vector $x=e_1=(1,0,\ldots,0)$, since for every other point with $|x|=1$ the result will follow by rotation. If $|x|>1$, one can consider the function
$$
g(y):=|x|^mF(|x|y)\geq F(y)
$$
and note that 
$$
M^-F(x)=CM^-g(x/|x|)\geq CM^-F(x/|x|)>0,
$$
for a constant $C>0$. In order to prove the lemma for $x=e_1$, we will use the following elementary inequalities:
\begin{equation}\label{4.1}
(a+b)^{-q}\geq a^{-q}\left(1-q\frac{b}{a}\right),
\end{equation}
\begin{equation}\label{4.2}
(a+b)^{-q}+(a-b)^{-q}\geq 2a^{-q}+q(q+1)b^2a^{-q-2},
\end{equation}
where $a>b>0$ and $q>0$. Using \eqref{4.1} and \eqref{4.2} for $|y|<\frac{1}{2}$, we get
\begin{eqnarray*}
\delta(F,e_1,y)&=&|e_1+y|^{-m}+|e_1-y|^{-m}-2\\ \nonumber
&=&\left(1+|y|^2+2y_1\right)^{-m/2}+\left(1+|y|^2-2y_1\right)^{-m/2}-2\\ \nonumber
&\geq&2\left(1+|y|^2\right)^{-m/2}+m(m+2)y_1^2\left(1+|y|^2\right)^{-m/2-2}-2\\ \nonumber
&\geq&m\left(-|y|^2+(m+2)y_1^2-\frac{1}{2}(m+2)(m+4)y_1^2|y|^2\right).
\end{eqnarray*}
We choose $m>0$ large enough to guarantee 
\begin{equation}\label{4.3}
(m+2)\lambda\int_{\partial S_1}y_1^2\,d\sigma(y)-\Lambda|\partial S_1|=:\delta_0>0.
\end{equation}
Then we make use of the above relation to estimate the part of the integral in $M^-F(e_1)$ over the set $S_r$ (with $r>0$ small). More precisely,

\begin{eqnarray*}
M^-F(e_1)&=&(2-\sigma)\int_{S_r}\frac{\lambda\delta^+-\Lambda\delta^-}{v_x^{\frac{n+\sigma}{2}}(y)}\,dy\\ 
& & +(2-\sigma)\int_{\R^n\setminus S_r}\frac{\lambda\delta^+-\Lambda\delta^-}{v_x^{\frac{n+\sigma}{2}}(y)}\,dy\\ 
&\geq&(2-\sigma)C\int_0^r\frac{\lambda m\delta_0s^2-\frac{1}{2}m(m+2)(m+4)C\Lambda s^4}{s^{n+\sigma}}\,ds\\ 
&&-(2-\sigma)\int_{\R^n\setminus S_r}\Lambda\frac{2^m}{v_x^{\frac{n+\sigma}{2}}(y)}\, dy\\ 
&\geq&cr^{2-\sigma}m\delta_0-m(m+2)(m+4)C\frac{2-\sigma}{4-\sigma}r^{4-\sigma}\\
&&-\frac{2-\sigma}{\sigma}C2^{m+1}r^{-\sigma},
\end{eqnarray*}
where $c$ and $C$ are positive constants (independent of $\sigma$). Note that we used \eqref{4.3} to bound the first integral and the fact that $0\leq F(x)\leq 2^m$ to bound the second. We finish the proof by choosing $\sigma_0$ close enough to 2, so that the factor $(2-\sigma)$ forces the second and the third terms in the last inequality to be very small to conclude
$$
M^-F(e_1)\geq\frac{1}{2}cr^{2-\sigma}m\delta_0>0.
$$
\end{proof}

Arguing as in Corollary 9.2 of \cite{CS09}, with the obvious adaptations, we obtain the following corollary.
 
\begin{corollary}\label{c4.1}
For any $\sigma_0\in(0,2)$ and $r>0$, there exist $m>0$ and $s>0$ such that the function
$$
F(x)=\min\left(s^{-m},|x|^{-m}\right)
$$
is a subsolution, i.e.,
$$
M^-F(x)\geq0,
$$
for all $\sigma\in(\sigma_0,2)$ and $|x|>r$, where the constants $m$ and $s$ depend only on $\lambda$, $\Lambda$ and the dimension.
\end{corollary}
\begin{corollary}\label{c4.2}
For given $r>0$, $R>1$ and $\sigma_0\in(0,2)$, there exist $s>0$ and $m>0$ such that the function
$$
g(x):=\min\left(s^{-m},|T_r^{-1}x|^{-m}\right)
$$
satisfies
$$
M^-g(x)\geq0,
$$
for $x\in\R^n\setminus S_r$, where $T_r$ is the normalization map of the section $S_r$.
\end{corollary}
\begin{proof}
Since 
$$
g(x)=F\left(T_r^{-1}x\right),\,\, \hbox{ for }\,\,x\in\R^n
$$
and
$$
M^-g(x)=C|\det T_r|M^-F\left(T_r^{-1}x\right)\geq0,
$$
for all $x\in\R^n\setminus S_r$, the result follows from Corollary \ref{c4.1}.
\end{proof}

We are now ready to construct the function which  will later be added to $u$ to force the contact set with $\Gamma$ to stay inside of the intermediate sections.
\begin{lemma}\label{l4.2}
For a given $\sigma_0\in(0,2)$, there exists a continuous function $\psi:\R^n\rightarrow\R$ satisfying the following conditions:
\begin{itemize}
\item $\psi=0$ in $\R^n\setminus S_{2\tau}$;
\item $\psi>2$ in $S_\tau$;
\item $M^-\psi>-\varphi$ in $\R^n$, for some positive function $\varphi$ supported in $\overline{S}_{1/4}$ for every $\sigma>\sigma_0$.
\end{itemize}
Above, $\tau>0$ is as in Proposition \ref{p3.1}.
\end{lemma}
\begin{proof}

We prove the lemma by constructing the function $\psi$. Let $s>0$ and $m>0$ be as in Corollary \ref{c4.1} with $r=1/4$. Set
$$
\frac{1}{c}\psi(x):=\left\{
\begin{array}{ll}
0\textrm{ in }\R^n\setminus S_{2\tau}\\
\\
 |T^{-1}_{\frac{1}{4}}x|^{-m}-(2\tau)^{-m}  \textrm{ in } S_{2\tau}\setminus S_s\\
 \\
q\textrm{ in } S_s,
\end{array}
\right.
$$
where $q$ is a quadratic paraboloid chosen such that $\psi$ is $C^{1,1}$ across $\partial S_s$. The constant $c$ is chosen such that $\psi>2$ in $S_\tau$. Since $\psi\in C^{1,1}(S_{2\tau})$, then from Lemma \ref{l2.3}, we have that $M^-\psi\in C(S_{2\tau})$. Corollary \ref{c4.1} then gives $M^-\psi\geq0$ in $\R^n\setminus S_{1/4}$, which completes the proof.
\end{proof}

\section{Towards the Harnack inequality}\label{s5}
In this section we prove a lemma which bridges the gap between a pointwise estimate and an estimate in measure. This is the main tool towards the proof of the Harnack inequality, as in \cite{CC95,CLU14,CS09}. It is here that we will make use of the function $\psi$ from Lemma \ref{l4.2}.
\begin{lemma}\label{l5.1}
Let $\sigma\in(0,2)$ and $\sigma_0\in(0,\sigma)$. There exist constants $\varepsilon_0>0$, $\eta\in(0,1)$ and $M>1$, depending only on $\sigma_0$, $\lambda$, $\Lambda$ and $n$, such that if, with $\tau>0$ as in Proposition \ref{l3.1}, 
$$
u\geq0  \mbox{ in } \R^n; \qquad \inf_{S_\tau}u\leq1; \qquad M^-u \leq \varepsilon_0 \mbox{ in } S_{2\tau},
$$
then
$$\left| \left\{ u\leq M\}\cap S_1 \right\} \right| > \eta.$$
\end{lemma}
\begin{proof}
Note that if $\sigma$ is far from $2$, one can prove the lemma adapting the ideas from \cite{S06}, but as in \cite{CS09} we argue differently  to guarantee an estimate that remains uniform as $\sigma\rightarrow2$. 

Define $\varrho :=\psi-u$, where $\psi$ is the function from Lemma \ref{l4.2}, and observe that
$$
M^+\varrho\geq M^-\psi-M^-u\geq-\varphi-\varepsilon_0.
$$
Let now $\Gamma$ be the concave envelope of $\varrho$ in $S_{4\tau}$. Applying Theorem \ref{t3.1} (rescaled) to $\varrho$, we get a family of sections $S_i$ such that
\begin{eqnarray*}
\max_{S_{2\tau}} \varrho&\leq& C|\nabla\Gamma(S_{2\tau})|^{1/n}\leq C\left(\displaystyle\sum_i|\nabla\Gamma(\overline{S}_i)|\right)^{1/n}\nonumber\\
&\leq&\left(C\displaystyle\sum_i(\max_{S_i}(\varphi+\varepsilon)^+)^n|S_i|\right)^{1/n}\nonumber\\
&\leq&C\varepsilon_0+C\left(\displaystyle\sum_i(\max_{S_i}(\varphi^+)^n|S_i|\right)^{1/n},
\end{eqnarray*}
with $C>0$ constant. On the other hand, we have $\displaystyle\max_{S_{2\tau}}\varrho\geq1$, since $\displaystyle\inf_{S_\tau}u\leq1$ and $\psi\geq2$ in $S_\tau$, and therefore
$$
1\leq C\varepsilon_0+C\left(\displaystyle\sum_i(\max_{S_i}(\varphi^+)^n|S_i|\right)^{1/n}.
$$
Hence, if $\varepsilon_0$ is small enough, one has 
$$
\frac{1}{2}\leq C\left(\displaystyle\sum_i(\max_{S_i}(\varphi^+)^n|S_i|\right)^{1/n}.
$$
Also, since $\supp\,\varphi\subset S_{1/4}$, 
$$
\frac{1}{2}\leq C\left(\sum_{S_i\cap S_{1/4}\neq\emptyset}|S_i|\right)^{1/n}
$$
or else
\begin{equation}\label{5.1}
\sum_{S_i\cap S_{1/4}\neq\emptyset}|S_i|\geq C.
\end{equation}
Also, the height of $S_i$ is bounded by $2^{-1/(2-\sigma)}<1$. Hence, every time $S_i$ intersects $S_{1/4}$, one has $\kappa S_i\subset S_{1/4}$, for $\kappa>0$ as in Theorem \ref{t3.1}. An application of Theorem \ref{t3.1} then gives
\begin{equation}\label{5.2}
\left| \left\{ y\in \kappa S_i \, : \, \varrho(y)\geq\Gamma(y)-Cr_i^2\right\}\right|\geq\mu\left|S_i\right|,
\end{equation}
and $Cr_i^2<C$. Observe that the family $\{\kappa S_i\}$, where $S_i\cap S_{1/4}\neq\emptyset$, is an open covering for $\bigcup_{i}\overline{S}_i$ and is contained in $S_{1/2}$. By taking a subcovering with bounded overlapping and using \eqref{5.1} and \eqref{5.2}, one gets
\begin{eqnarray*}
&&\left| \left\{ y\in S_{1/2}\,:\, \varrho(y)\geq\Gamma(y)-C \right\} \right| \\
&\geq&\left|\bigcup_{i}\{y\in\kappa S_i\,:\,\varrho(y)\geq\Gamma(y)-C\}\right|\\
&\geq&C_1\sum_{i} \left| \left\{ y\in\kappa S_i\,:\,\varrho(y)\geq\Gamma(y)-C \right\} \right| \\
&\geq&C_1c_1.
\end{eqnarray*}
Therefore, if $l:=\displaystyle\max_{S_{1/2}}\psi$, then
$$
|\{y\in S_{1/2}\,:\,u(y)\leq l+C\}|\geq C_1c_1.
$$
Hence, for $M:=l+C$, noting that $S_{1/2}\subset S_1$, one has
$$
|\{y\in S_1\,:\, u(y)\leq M\}|\geq C_1c_1,
$$
which completes the proof.
\end{proof}

As a consequence, using a variant of Calder\'on-Zygmund decomposition (Lemma \ref{l2.1}), as in \cite[Theorem 3]{CG97}, from Lemma \ref{l5.1} we get the following result.
\begin{theorem}\label{t5.1}
Let $z\in\R^n$, $u\geq0$ in $\R^n$, $\displaystyle\inf_{S_r(z)}u\leq1$, $M^-u\leq\varepsilon_0$ in $S_{2\tau}(z)$. There exist constants $\rho\in(0,1)$, $C>0$ and $\varepsilon>0$ such that
$$
|\{u>t\}\cap S_\rho(z)|\leq Ct^{-\varepsilon}|S_r(z)|,\,\,\forall t>0.
$$
Here $C$ and $\varepsilon>0$ depend only on $\lambda$, $\Lambda$, $g_-$, $g^+$ and $n$.
\end{theorem}
\begin{proof}
	As in \cite[Theorem 4]{CG97}, it is enough to consider the case when section $S_r(z)$ is normalized and has unit parameter $r=1$. Then the result follows from the fact that, as in \cite[proof of Theorem 4]{CG97}, using Lemma \ref{l5.1} and Lemma \ref{l2.1}, for $Q$ and $P$ large enough, one can construct a decreasing family of sections $S_k:=S_{r_k}(z)$, $k\in\mathbb{N}\cup \{0\}$, with $1=r_0>r_1>\ldots$,
	such that
	$$
	\left| \left\{ u>QP^{k+2} \right\}\cap S_{k+1}\right| < c(\theta) \left| \left\{ u>QP^{k+1} \right\} \cap S_k \right|,\,\,k\in\mathbb{N},
	$$
	where $c(\theta)<1$ is as in Lemma \ref{l2.1}, 
	$$
	r_k=1-\sum_{i=1}^k\left(\gamma\left(\frac{H}{QP^{i+1}}\right)^{\frac{1}{\varepsilon}}\right)^{\frac{1}{p}},
	$$
	$p>1$, and $H$ is a structural constant. Passing to the limit as $k\rightarrow\infty$, we obtain
	$$
	\rho:=1-\sum_{i=1}^\infty\left(\gamma\left(\frac{H}{QP^{i+1}}\right)^{\frac{1}{\varepsilon}}\right)^{\frac{1}{p}}\in(0,1),
	$$	
	once $P$, $Q$ are large enough.
	
	Since $c(\theta)<1$, Lemma \ref{l2.1} then implies the result for the section $S_\rho(z)$. 
\end{proof}

\section{The Harnack inequality and consequences}\label{s6}
In this section, we prove the Harnack inequality for integro-differential equations with kernels deforming like sections of a strictly convex solution to a Monge-Amp\`ere equation and, as a consequence, we derive $C^\alpha$ and $C^{1,\alpha}$ estimates for solutions. The Harnack inequality remains uniform as $\sigma\rightarrow2$. We need the following auxiliary result.
\begin{theorem}\label{t6.1}
Let $\sigma_0>0$, $\sigma\geq\sigma_0$ and $C_0>0$. If $u\geq0$ in $\R^n$, $M^-u\leq C_0$, $M^+u\geq-C_0$ in $S_{2\tau}$, then there exists $C>0$, depending on $\sigma_0$ but not on $\sigma$, such that
$$
u(x)\leq C(u(0)+C_0)\,\textrm{ in }\,S_{\rho/2},
$$
where $\rho$ is as in Theorem \ref{t5.1}.
\end{theorem}
\begin{proof}
Without loss of generality, one can assume that $u(0)\leq1$ and $C_0=1$ 
(otherwise divide by $u(0)+C_0$). Take $\varepsilon>0$ as in Theorem \ref{t5.1} and set $\kappa=\frac{n}{2\varepsilon}$ and 
$$
v_\theta(x):=\theta(dist(x,\partial S_1))^{-\kappa},\,\,\forall x\in S_1.
$$
Let now $\theta_0>0$ be the minimum value of $\theta$ for which there holds $u\leq v_\theta$ in $S_1$. Note that there must be a point $x_0\in S_1$ such that $u(x_0)=v_{\theta_0}(x_0)$ (otherwise one would be able to take $\theta_0$ smaller). As in \cite{CS09}, the aim is to show that $\theta_0$ can not be too large, i.e., that there exists $C>0$ such that $\theta_0<C$. 

For that purpose, we estimate the portion of the section $S_r(x_0)$ covered by $\{u<u(x_0)/2\}$ and by $\{u>u(x_0)/2\}$, where $r=d/2$, $d$ being the distance of the point $x_0$ to $\partial S_1$. Theorem \ref{t5.1} provides
$$
\left|\left\{u>\frac{u(x_0)}{2}\right\}\cap S_\rho\right|\leq C\left|\frac{2}{u(x_0)}\right|^\varepsilon=C2^\varepsilon\theta^{-\varepsilon}d^{\kappa\varepsilon}\leq C_1\theta_0^{-\varepsilon}r^n.
$$
On the other hand, $|S_r(x_0)|\geq C_2r^n$, so 
\begin{equation}\label{6.1}
\left|\left\{u>\frac{u(x_0)}{2}\right\}\cap S_r(x_0)\right|\leq\frac{C_1}{C_2}\theta_0^{-\varepsilon}|S_r(x_0)|, 
\end{equation}
which means that if $\theta_0$ is large, then the set $\{u>u(x_0)/2\}$ can cover only a small portion of $S_r(x_0)$.

Our next task is to show that if $\theta_0$ is large, then the measure of the portion of $S_r(x_0)$ covered by $\{u<u(x_0)/2\}$ does not exceed $(1-\delta)|S_r(x_0)|$, for a positive constant $\delta$ independent of $\theta_0$. This will lead to a contradiction, hence $\theta_0$ must be bounded, and the result will follow. 

Let $h>0$ be so small that
$$
dist(x,\partial S_1)\geq d-\frac{hd}{2},\,\,\forall x\in S_{hr}(x_0),
$$
and so, for every $x\in S_{hr}(x_0)$, one has
$$
u(x)\leq v_{\theta_0}(x)\leq\theta_0\left(d-\frac{hd}{2}\right)^{-\kappa}\leq u(x_0)\left(1-\frac{h}{2}\right)^{-\kappa}.
$$
Therefore, 
$$
\omega(x):=\left(1-\frac{h}{2}\right)^{-\kappa}u(x_0)-u(x)\geq0\quad \textrm{ in }\,S_{hr}(x_0),
$$
and $M^-\omega\leq1$. The latter follows from the fact that $M^+u\geq-1$. We would like to apply Theorem \ref{t5.1} (rescaled)  to $\omega$, but we can not do so because $\omega$ is not non-negative in the whole space, but just in $S_{hr}$. This leads us to consider the function $a:=\omega^+$ instead, and estimate the change in the right hand side due to the truncation error. We need to find an estimate for $M^-a$ from above.
For $x\in\R^n$, we have
\begin{eqnarray}\label{6.2}
\frac{M^-a(x)-M^-\omega(x)}{2-\sigma}&=&\lambda\int_{\R^n}\frac{\delta^+(a,x,y)-\delta^+(\omega,x,y)}{v_x^{\frac{n+\sigma}{2}}(y)}\,dy\nonumber\\
&+&\Lambda\int_{\R^n}\frac{\delta^-(\omega,x,y)-\delta^-(a,x,y)}{v_x^{\frac{n+\sigma}{2}}(y)}\,dy\nonumber\\
&:=&I_1+I_2.
\end{eqnarray}
Note that if $\delta_g:=\delta(g,x,y)$, then
$$
\delta_a^+=\delta_\omega+\omega^-(x-y)+\omega^-(x+y)
$$
due to the elementary equality
$$
\omega^+(x+y)=\omega(x+y)+\omega^-(x+y).
$$
Also,
$$
\delta_a^+\geq\delta_\omega^+\,\,\textrm{ and }\,\,\delta_\omega=\delta_\omega^+-\delta_\omega^-,
$$
so we estimate
\begin{eqnarray}\label{6.3}
I_1&=&-\lambda\int_{\{\delta_a^+>\delta_\omega^+\}}\frac{\delta_\omega^-}{v_x^{\frac{n+\sigma}{2}}(y)}\,dy\nonumber\\
&+&\lambda\int_{\{\delta_a^+>\delta_\omega^+\}}\frac{\omega^-(x+y)+\omega^-(x-y)}{v_x^{\frac{n+\sigma}{2}}(y)}\,dy\nonumber\\
&\leq&\Lambda\int_{\{\delta_a^+>0\}}\frac{\omega^-(x+y)+\omega^-(x-y)}{v_x^{\frac{n+\sigma}{2}}(y)}\,dy.
\end{eqnarray}
Similarly,
\begin{eqnarray}\label{6.4}
I_2&=&\Lambda\int_{\{\delta_\omega^->0\}\cap\{\delta_a^-\neq\delta_\omega^-\}}\frac{\delta_\omega^--\delta_a^-}{v_x^{\frac{n+\sigma}{2}}(y)}\,dy\nonumber\\
&+&\Lambda\int_{\{\delta_\omega^-=0\}\cap\{\delta_a^-\neq\delta_\omega^-\}}\frac{\omega^-(x+y)+\omega^-(x-y)}{v_x^{\frac{n+\sigma}{2}}(y)}\,dy\nonumber\\
&\leq&\Lambda\int_{\{\delta_\omega^->0\}\cap\{\delta_a^-\neq\delta_\omega^-\}}\frac{-\delta_\omega-\delta_\omega^-}{v_x^{\frac{n+\sigma}{2}}(y)}\,dy.
\end{eqnarray}
Observe that
\begin{eqnarray}\label{6.5}
-\delta_\omega^--\delta_a^-&=&2\omega(x)-(\omega(x+y)+\omega(x-y))-\delta_a^-\nonumber\\
&=&2\omega(x)-[(\omega^+(x+y)+\omega^+(x-y))\nonumber\\
&\,\,&-(\omega^-(x+y)+\omega^-(x-y))]\nonumber\\
&=&-\delta_a-\delta_a^-+\omega^-(x+y)+\omega^-(x-y)\nonumber\\
&=&-\delta_a^++\omega^-(x+y)+\omega^-(x-y).
\end{eqnarray}
Using \eqref{6.4} and \eqref{6.5} we then get
\begin{eqnarray}\label{6.6}
I_2&\leq&-\Lambda\int_{\{\delta_\omega^->0\}\cap\{\delta_a^-\neq\delta_\omega^-\}}\frac{\delta_a^+}{v_x^{\frac{n+\sigma}{2}}(y)}\,dy\nonumber\\
&+&\Lambda\int_{\{\delta_\omega^->0\}\cap\{\delta_a^-\neq\delta_\omega^-\}}\frac{\omega^-(x+y)+\omega^-(x-y)}{v_x^{\frac{n+\sigma}{2}}(y)}\,dy\nonumber\\
&\leq&\Lambda\int_{\{\delta_a^-\geq0\}}\frac{\omega^-(x+y)+\omega^-(x-y)}{v_x^{\frac{n+\sigma}{2}}(y)}\,dy.
\end{eqnarray}
Therefore, from \eqref{6.2}, \eqref{6.3} and \eqref{6.6}, one gets
\begin{eqnarray*}
\frac{M^-a(x)-M^-\omega(x)}{2-\sigma}&\leq&\Lambda\int_{\R^n}\frac{\omega^-(x+y)+\omega^-(x-y)}{v_x^{\frac{n+\sigma}{2}}(y)}\,dy\nonumber\\
&=&-2\Lambda\int_{\{\omega(x+y)<0\}}\frac{\omega(x+y)}{v_x^{\frac{n+\sigma}{2}}(y)}\,dy.
\end{eqnarray*}
Moreover, by the definition of $\omega$, for $x\in S_{hr/2}(x_0)$ we have
\begin{eqnarray*}
&&\frac{M^-a(x)-M^-\omega(x)}{2-\sigma}\leq2\Lambda\int_{\omega(x+y)<0}\frac{-\omega(x+y)}{v_x^{\frac{n+\sigma}{2}}(y)}\,dy\nonumber\\
&\leq&2\Lambda\int_{\R^n\setminus S_{hr}(x_0-x)}\frac{\left(u(x+y)-\left(1-\frac{h}{2}\right)^{-\kappa}u(x_0)\right)^+}{v_x^{\frac{n+\sigma}{2}}(y)}\,dy.
\end{eqnarray*}
Observe that if $t>0$ is the largest value for which $u(x)\geq t(1-|4x|^2)$, then there must be a point $x_1$ in a smaller section $S_r$ such that $u(x_1)=1-|4x_1|^2$. Since $u(0)\leq1$, then $t\leq1$. Thus,
$$
(2-\sigma)\int_{\R^n}\frac{\delta^-(u,x_1,y)}{v_x^{\frac{n+\sigma}{2}}(y)}\,dy\leq(2-\sigma)\int_{\R^n}\frac{\delta^-(1-|4x_1|^2,x_1,y)}{v_x^{\frac{n+\sigma}{2}}(y)}\,dy\leq C,
$$
where the constant $C>0$ does not depend on $\sigma$. On the other hand, since $M^-u(x_1)\leq1$, we find
$$
(2-\sigma)\int_{\R^n}\frac{\delta^+(u,x_1,y)}{v_x^{\frac{n+\sigma}{2}}(y)}\,dy\leq C.
$$
In particular, since $u(x_1)\leq1$ and $u(x_1-y)\geq0$, we have
$$
(2-\sigma)\int_{\R^n}\frac{(u(x_1+y)-2)^+}{v_x^{\frac{n+\sigma}{2}}(y)}\,dy\leq C.
$$
By assuming $\theta_0>0$ is large enough, we can suppose that $u(x_0)>2$.
Writing
$$
u(x+y)-\left(1-\frac{h}{2}\right)^{-\kappa}u(x_0)=u(x+x_1+y-x_1)-\left(1-\frac{h}{2}\right)^{-\kappa}u(x_0),
$$ 
we estimate
\begin{align*}
&2\Lambda(2-\sigma)\int_{\R^n\setminus S_{hr}(x_0-x)}\frac{\left(u(x+y)-\left(1-\frac{h}{2}\right)^{-\kappa}u(x_0)\right)^+}{v_x^{\frac{n+\sigma}{2}}(y)}\,dy\\
&\leq2\Lambda(2-\sigma)\int_{\R^n\setminus S_{hr/2}(x_0-x)}\frac{\left(u(x_1+y+x-x_1)-\left(1-\frac{h}{2}\right)^{-\kappa}u(x_0)\right)^+}{v_x^{\frac{n+\sigma}{2}}(y+x-x_1)}\,\\
&\quad \cdot\frac{v_x^{\frac{n+\sigma}{2}}(y+x-x_1)}{v_x^{\frac{n+\sigma}{2}}(y)}\,dy\\
&\leq C(hr)^{-\frac{n+\sigma}{2}}.
\end{align*}
Hence, since
$$
M^-a=M^-\omega+(M^-a-M^-\omega),
$$
we finally conclude
$$
M^-a\leq 1+C(hr)^{-\frac{n+\sigma}{2}}\,\,\textrm{ in }\,\,S_{hr/2}(x_0),
$$
where the constant $C>0$ does not depend on $\sigma$. This allows to apply Theorem \ref{t5.1} to $a$ in $S_{hr/2}(x_0)$. From Theorem \ref{t5.1} and the fact that
$$
a(x_0)=\left(\left(1-\frac{h}{2}\right)^{-\kappa}-1\right)u(x_0),
$$
one has
\begin{align}\label{6.7}
&\left|\left\{u<\frac{u(x_0)}{2}\right\}\cap S_{hr/4}(x_0)\right|\nonumber\\
&=\left|\left\{a(x)>u(x_0)\left(\left(1-\frac{h}{2}\right)^{-\kappa}-\frac{1}{2}\right)\right\}\cap S_{hr/4}(x_0)\right|\nonumber\\
&\leq C \left| S_{hr/4}(x_0) \right |\left[ \left( \left(1-\frac{h}{2}\right)^{-\kappa}-1\right)u(x_0) \right. \nonumber\\
& \quad \left. + \left( 1+C(hr)^{-\frac{n+\sigma}{2}} \right) (rh)^\sigma \right]^\varepsilon\left(u(x_0)\left(\left(1-\frac{h}{2}\right)^{-\kappa}-\frac{1}{2}\right)\right)^{-\varepsilon}\nonumber\\
&\leq C \left| S_{hr/4}(x_0) \right| \left[\left(\left(1-\frac{h}{2}\right)^{-\kappa}-1\right)u(x_0)+C_1(hr)^{-c(n)}\right]^\varepsilon\nonumber\\
&\quad \cdot\left(u(x_0)\left(\left(1-\frac{h}{2}\right)^{-\kappa}-\frac{1}{2}\right)\right)^{-\varepsilon}\nonumber\\
&\leq C \left| S_{hr/4}(x_0) \right| \left[\left(\left(1-\frac{h}{2}\right)^{-\kappa}-1\right)^\varepsilon+h^{-c(n)\varepsilon}t^{-\varepsilon}\right],
\end{align}
where $c(n)>0$ does not depend on $\sigma$. In order to get the last estimate in \eqref{6.7}, we used the inequalities
\begin{eqnarray*}
&&\left[\left(\left(1-\frac{h}{2}\right)^{-\kappa}-\frac{1}{2}\right)u(x_0)+C_1(hr)^{-c(n)}\right]^\varepsilon\nonumber\\
&\leq&\left(\left(1-\frac{h}{2}\right)^{-\kappa}-\frac{1}{2}\right)^\varepsilon u^\varepsilon(x_0)+C_1(hr)^{-c(n)\varepsilon}
\end{eqnarray*}
and
$$
\left(1-\frac{h}{2}\right)^{-\kappa}-\frac{1}{2}\geq\left(1-\frac{h}{2}\right)^{-\frac{n}{\varepsilon}}-\frac{1}{2}\geq\frac{1}{2},
$$
for $h>0$ sufficiently small, and also
\begin{eqnarray*}
&&C_2h^{-c(n)\varepsilon}u^{-\varepsilon}(x_0)\left(\left(1-\frac{h}{2}\right)^{-\kappa}-\frac{1}{2}\right)^{-\varepsilon}\nonumber\\
&\leq&C_3h^{-c(n)\varepsilon}r^{-c(n)\varepsilon}u^{-\varepsilon}(x_0)\leq C_4h^{-c(n)\varepsilon}\theta_0^{-\varepsilon}d^{n(1-c\varepsilon)}\leq C_5h^{-c(n)\varepsilon}\theta_0^{-\varepsilon}.
\end{eqnarray*}
We choose $h>0$ small enough to guarantee
\begin{eqnarray}\label{6.8}
&&C|S_{hr/4}(x_0)|\left(\left(1-\frac{h}{2}\right)^{-\kappa}-1\right)^\varepsilon\nonumber\\
&\leq& C|S_{hr/4}(x_0)|\left(\left(1-\frac{h}{2}\right)^{-\frac{2n}{\varepsilon}}-1\right)^\varepsilon\nonumber\\
&\leq&\frac{1}{4}|S_{hr/4}(x_0)|.
\end{eqnarray}
Observe that we can choose such $h$ independently of $\theta_0$. Then, for this fixed $h$, we take $\theta_0>0$ large enough to guarantee
\begin{equation}\label{6.9}
C|S_{hr/4}(x_0)|h^{-c(n)\varepsilon}\theta_0^{-\varepsilon}\leq\frac{1}{4}|S_{hr/4}(x_0)|.
\end{equation}
Combining \eqref{6.7}-\eqref{6.9}, we obtain
$$
\left|\left\{u<\frac{u(x_0)}{2}\right\}\cap S_{hr/4}(x_0)\right|\leq\frac{1}{4}|S_{hr/4}(x_0)|,
$$
which implies, for $\theta_0>0$ large,
$$
\left|\left\{u>\frac{u(x_0)}{2}\right\}\cap S_{hr/4}(x_0)\right|\geq c |S_r(x_0)|,
$$
which contradicts \eqref{6.1}.
\end{proof}

As a consequence of the Harnack inequality, we obtain the H\"{o}lder regularity of solutions.
\begin{theorem}\label{t6.2}
Let $\sigma_0>0$ and $\sigma\in(\sigma_0,2)$. If $u$ is a bounded function in $\R^n$ such that
$$
M^-u\leq C_0\,\,\textrm{ and }\,\,M^+u\geq-C_0\,\,\textrm{ in }\,\,S_{2\tau}(x_0),
$$
then there exists a positive constant $\alpha\in(0,1)$, depending only on $\lambda$, $\Lambda$, $\sigma_0$ and dimension, such that $u\in C^\alpha(S_{\rho/2})$ and
$$
\|u\|_{C^\alpha(S_{\rho/2}(x_0))}\leq C \left( \sup_{\R^n}|u|+C_0 \right),
$$
for a constant $C>0$, depending only on the norm of the normalization map that normalizes the section $S_{\gamma\rho/2}(x_0)$, $\lambda$, $\Lambda$, $\sigma_0$, $C_0$ and dimension. Here the constant $\tau>1$ is as in Proposition \ref{p3.1} and $\rho\in(0,1)$ as in Theorem \ref{t5.1}, and $\gamma>1$ is the engulfing constant.
\end{theorem}

\begin{proof}
	From the Harnack inequality, as in \cite[Lemma 8.23]{GT83}, we conclude that there exist $C>0$ and $\alpha\in(0,1)$ such that
	$$
	\osc_{S_r({x_0})}u\le C\left(\frac{r}{\rho}\right)^{\alpha}\max_{S_\rho({x_0})}u,\,\,\,\,r<\rho.
	$$
	Let $x$, $y\in S_{\rho/2}(x_0)$ and $S_{r_0}(x)$ be the smallest section containing $y$. By the engulfing property, Theorem \ref{t2.2}, $S_{\rho/2}(x_0)\subset S_{\gamma\rho/2}(x)$. Thus, $y\in S_{\gamma\rho/2}(x)$, hence $r_0\le\gamma\rho/2$, since $S_{r_0}(x)$ is the smallest section containing $y$. If $T$ is the affine transformation that normalizes the section $S_{\gamma\rho/2}(x_0)$, then arguing as in \cite[Section 4]{CG97}, we have
	$$
	r_0\le 2\gamma\rho\left(\frac{2\|T\|}{K_2}\right)^{1/\varepsilon}|x-y|^{1/\varepsilon},
	$$
	where $K_2>0$ and $\varepsilon>0$ are the constants appearing in Theorem \ref{t2.3}. Thus,
	$$
	|u(x)-u(y)|\le\osc_{S_{r_0}(x)}u\le\left(\frac{r_0}{\rho}\right)^{\alpha}\max_{S_\rho({x})}u\le C|x-y|^{\alpha/\varepsilon}\max_{S_\rho({x})}u.
	$$
\end{proof}

In order to prove the interior $C^{1,\alpha}$ regularity of solutions one needs to have an extra assumption on the kernels. The idea is to use Theorem \ref{t6.2} for incremental quotients of the solution, but since we do not have a uniform bound in $L^\infty$ for these incremental quotients outside of the domain, we assume a modulus of continuity in measure for the kernel, to make sure that faraway oscillations tend to cancel out. More precisely, for a given $\varrho>0$, we define the class $\mathcal{L}_1$ of the operators $L$ with kernels $K$ satisfying not only \eqref{1.3}, but additionally
\begin{equation}\label{6.10}
\int_{\R^n\setminus S_\varrho}\frac{K_x^{\alpha\beta}(y)-K_x^{\alpha\beta}(y-h)}{|h|}\,dy\leq \Upsilon,\,\,\,\,\textrm{for}\,\,\,\,|h|<\frac{\varrho}{2}.
\end{equation}
The proof of the next theorem is essentially the same as the one of Theorem 13.1 of \cite{CS09}, hence we will omit it.

\begin{theorem}\label{t6.3}
Let $\sigma_0>0$ and $\sigma\in(\sigma_0,2)$. Let also the kernels $K_x^{\alpha\beta}$ satisfy $\eqref{1.3}$ and \eqref{6.10}. If $u$ is a bounded function such that $Iu=f$ in $S_{2\tau}$, then there is a constant $\gamma\in(0,1)$, depending only on $\lambda$, $\Lambda$, $\sigma_0$ and dimension, such that $u\in C^{1,\gamma}(S_{\rho/2})$ and
$$
\|u\|_{C^{1,\gamma}(S_{\rho/2})}\leq C\sup_{\R^n}|u|,
$$
for a constant $C>0$, depending only on the norm of the normalization map of the section $S_{\gamma\rho/2}(x_0)$, $\lambda$, $\Lambda$, $\sigma_0$, $\Upsilon$ and dimension. Here the constant $\tau>1$ is as in Proposition \ref{p3.1} and $\rho\in(0,1)$ as in Theorem \ref{t5.1}.
\end{theorem}

\medskip

\textbf{Acknowledgments.} We thank the anonymous referees for their comments and suggestions that helped improving an earlier version of the paper.


\begin{thebibliography}{99}

\bibitem{AFT98} 
H. Aimar, L. Forzani and R. Toledano, 
\textit{Balls and quasi-metrics: a space of homogeneous type modelling the real analysis related to the Monge-Amp\`ere equation}, 
J. Fourier Anal. Appl. 4 (1998), 377--381.

\bibitem{BL02} R.F. Bass and D.A. Levin, \textit{Harnack inequalities for jump processes}, Potential Anal. 17 (2002), 375--388.

\bibitem{C901} L.A. Caffarelli, 
\textit{A localization property of viscosity solutions to the Monge-Amp\`ere equation and their strict convexity}, 
Ann. of Math. (2) 131 (1990), 129--134. 
 
\bibitem{C902} L.A. Caffarelli, 
\textit{Interior $W^{2,p}$ estimates for solutions of the Monge-Amp\`ere equation}, 
Ann. of Math. (2) 131 (1990), 135--150. 

\bibitem{C91} L.A. Caffarelli, 
\textit{Some regularity properties of solutions of Monge-Amp\`ere equation}, 
Comm. Pure Appl. Math. 44 (1991), 965--969. 

\bibitem{CC95} L.A. Caffarelli and X. Cabr\'e, 
\textit{Fully nonlinear elliptic equations}, 
AMS Colloquium Publications 43, American Mathematical Society, Providence, R.I., 1995.

\bibitem{CG96} L.A. Caffarelli and C.E. Guti\'errez, 
\textit{Real analysis related to the Monge-Amp\`ere equation}, 
Trans. Amer. Math. Soc. 348 (1996), 1075--1092.

\bibitem{CG97} L.A. Caffarelli and C.E. Guti\'errez, 
\textit{Properties of the solutions of the linearized Monge-Amp\`ere equation}, 
Amer. J. Math. 119 (1997), 423--465.

\bibitem{CLU14} L.A. Caffarelli, R. Leit\~{a}o and J.M. Urbano, 
\textit{Regularity for anisotropic fully nonlinear integro-differential equations}, 
Math. Ann. 360 (2014), 681--714.

\bibitem{CS09} L.A. Caffarelli and L. Silvestre, 
\textit{Regularity theory for fully nonlinear integro-differential equations}, Comm. Pure Appl. Math. 62 (2009), 597--638.

\bibitem{CL12} H. Chang Lara, \textit{Regularity for fully nonlinear equations with non local drift}, arXiv:1210.4242v2.	

\bibitem{F17} A. Figalli, 
\textit{The Monge-Amp\`ere equation and its applications}, 
Zurich Lectures in Advanced Mathematics, European Mathematical Society (EMS), Zurich, 2017.

\bibitem{GT83} D. Gilbarg and N.S. Trudinger, \textit{Elliptic partial differential equations of second order}, Springer-Verlag, New York, 1983.

\bibitem{G16} C.E. Guti\'errez, 
\textit{The Monge-arXiv:1210.4242v2 equation}, 
2nd edition, Progress in Nonlinear Differential Equations and their Applications 89, Birkh\"{a}user, 2016.

\bibitem{GH00} C.E. Guti\'errez and Q. Huang, 
\textit{Geometric properties of the sections of solutions to the Monge-Amp\`ere equation}, 
Trans. Amer. Math. Soc. 352 (2000), 4381--4396.

\bibitem{I89} H. Ishii, 
\textit{On uniqueness and existence of viscosity solutions of fully nonlinear second-order elliptic PDEs}, 
Comm. Pure Appl. Math. 42 (1989), 15--45.

\bibitem{MS17} D. Maldonado and P.R. Stinga, 
\textit{Harnack inequality for the fractional nonlocal linearized Monge-Amp\`{e}re equation}, 
Calc. Var. Partial Differential Equations 56 (2017), Art. 103, 45 pp.
 
\bibitem{OS07}  B. \O ksendal and A. Sulem,  
\textit{Applied stochastic control of jump diffusions}, 2nd edition, Universitext, Springer, Berlin, 2007.

\bibitem{SS16} R.W. Schwab and L. Silvestre, \textit{Regularity for parabolic integro-differential equations with very irregular kernels}, Anal. PDE 9 (2016), 727--772.

\bibitem{S06} L. Silvestre, 
\textit{H\"{o}lder estimates for solutions of integro-differential equations like the fractional Laplace}, Indiana Univ. Math. J. 55 (2006), 1155--1174.
\end{thebibliography}
\end{document}